\documentclass[reqno,11pt]{amsart}
\usepackage[utf8]{inputenc}
\usepackage[T1]{fontenc}
\usepackage[oldstylenums]{kpfonts}
\usepackage{amsmath}
\usepackage{amssymb}
\usepackage{mathrsfs}
\usepackage{comment}
\usepackage{graphicx}
\usepackage[shortlabels]{enumitem}
\usepackage[margin=3.3cm]{geometry}
\usepackage[backend = biber,citestyle=ieee-alphabetic, bibstyle=alphabetic ,maxnames=4,minnames=3,maxbibnames=99,backref=true,hyperref=true]{biblatex}
\usepackage{bbm} 
\usepackage{enumerate} 
\usepackage[colorlinks=true,linkcolor=blue,citecolor=blue,urlcolor=blue]{hyperref} 
\usepackage[foot]{amsaddr}

\newtheorem{theorem}{Theorem}[section]
\newtheorem{proposition}[theorem]{Proposition}
\newtheorem{lemma}[theorem]{Lemma}
\newtheorem{conj}[theorem]{Conjecture}
\theoremstyle{definition}

\newtheorem{rmk}[theorem]{Remark}

\newcommand{\eq}[1]{\begin{equation}#1\end{equation}}
\newcommand{\eqn}[1]{\begin{equation*}#1\end{equation*}}
\newcommand{\al}[1]{\begin{align}#1\end{align}}
\newcommand{\aln}[1]{\begin{align*}#1\end{align*}}
\newcommand{\R}{\mathbb{R}}
\newcommand{\Z}{\mathbb{Z}}
\newcommand{\N}{\mathbb{N}}
\newcommand{\E}{\mathbb{E}}
\renewcommand{\P}{\mathbb{P}}
\newcommand{\var}{\mathrm{Var}}
\newcommand{\g}{\gamma}
\newcommand{\sqn}{{\square_n}}
\newcommand{\e}{\eta}
\renewcommand{\epsilon}{\varepsilon}
\newcommand{\G}{\Gamma}
\newcommand{\inv}[1]{\frac{1}{#1}}

\newcommand{\ens}[1]{\left\{#1\right\}}
\newcommand{\pr}[1]{\left( #1 \right)}
\newcommand{\crch}[1]{\left[ #1 \right]}
\newcommand*\dif{\mathop{}\!\mathrm{d}}
\newcommand{\law}{\mathrm{Law}}
\newcommand{\Leb}{\mathsf{Leb}}
\newcommand{\Poi}{\mathsf{Poi}}

\newcommand{\Binom}{\mathsf{Bin}}
\newcommand{\id}{\mathsf{id}}
\newcommand{\til}{\mathsf{til}}
\newcommand{\Ent}[2]{\mathsf{Ent}\left (#1~\middle\Vert~#2\right)}
\newcommand{\Ents}[2]{\mathscr{Ent}_s\left (#1~\middle\Vert~#2\right)}
\newcommand{\trib}{\mathscr{F}}

\newcommand{\sfP}{\mathsf{P}}

\newcommand{\sfR}{\mathsf{R}}
\newcommand{\sfC}{\mathsf{C}}
\newcommand{\sfW}{\mathsf{W}}
\newcommand{\Wass}{\mathsf{Wass}}

\newcommand{\pp}{\mathscr{P}(\Gamma)}
\newcommand{\pps}{\mathscr{P}_s(\Gamma)}

\newcommand{\abs}[1]{\left|#1\right|}

\title{Rigidity of One-Dimensional Point Processes via Optimal Transport}

\author{David Dereudre}
\address{Université de Lille, Laboratoire Paul Painlevé}
\email{david.dereudre@univ-lille.fr}
\author{Rafaël Digneaux}
\email{rafael.digneaux@univ-lille.fr}
\keywords{Stationary point processes, 1D, Rigidity, Number-Rigidity, Cyclic-Factor, Non-singular Riesz gas, Coulomb gas, Optimal transport. }

\begin{document}
	
\maketitle

	\begin{abstract}
We investigate rigidity phenomena in one-dimensional point processes. We show that the existence of an $L^1$ transport map from a stationary lattice or the Lebesgue measure to a point process is sufficient to guarantee the properties of Number-Rigidity and Cyclic-Factor. We then apply this result to non-singular Riesz gases with parameter $s\in(-2,-1]$, defined in infinite volume as accumulation points of stationarized finite-volume Riesz gases. This includes, for $s=-1$, the well-known one-dimensional Coulomb gas (also called Jellium plasma, or the one-component 1D plasma).

	\end{abstract}

	\section*{Introduction}
	Spatial point processes have attracted great interest across many scientific fields, ranging from spatial data analysis, materials science, telecommunications, and biology to statistical physics. Point processes form one of the central topics of stochastic geometry. In mathematical physics, they naturally arise when studying systems composed of a very large number of particles, such as a gas of electrons (a plasma). They also arise in more theoretical contexts, such as the study of eigenvalues of random matrices or the zeros of random polynomials, of finite or infinite degree. Unlike the classical Poisson point process, which exhibits perfect independence between particles, most interesting models display strong long-range dependencies. These dependencies usually entail some form of rigidity of the system, which accounts for physical behaviour such as incompressibility, screening effects, and crystallization.  
	
		In the mathematical community, rigidity properties have sparked a great deal of research. One such property is \textbf{Number-Rigidity}, the striking phenomenon whereby a point process has a fixed number of points in any compact set once the outside configuration is known. This property holds for i.i.d. perturbed lattices in dimensions $d=1,2$ with $L^d$ perturbations \cite{holroyd2013insertiondeletiontolerancepoint}, in dimension $d\geq3$ with Gaussian perturbations of sufficiently small variance \cite{peres2014rigidity}, and for several other models, including the Ginibre ensemble \cite{Ghosh_2017}, the determinantal point process with sine kernel \cite{ghosh2014determinantalprocessescompletenessrandom}, the Sine$_\beta$ gas \cite{dereudre2019dlrequationsrigiditysinebeta}, and the two-dimensional Coulomb gas \cite{leblé2024dlrequationsnumberrigiditytranslationinvariance}. Riesz gases for $s\in(d-1,d)$ are not Number-Rigid \cite{dereudre2021numberrigidity}, and furthermore, in dimension $d\geq3$, the Riesz gas for $s=d-2$ (the Coulomb gas) is not Number-Rigid \cite{thoma2023nonrigiditypropertiescoulombgas}. This property was first observed in \cite{AizenmanMartin}, which showed that the one-dimensional Coulomb gas is Number-Rigid. There also
	exist stronger rigidities, for instance the zeros of the Gaussian Analytic Function (GAF) exhibit center-of-mass rigidity, meaning the center of mass of points in a compact set is fixed once the outside configuration is known \cite{Ghosh_2017}. More generally, one can define rigidity at any order and have examples of such processes, see \cite{lachièzerey2025rigidityrandomstationarymeasures, kiro2019rigidityzerosetsgaussian}.
	
	Another property is \textbf{Hyperuniformity}, which is when the variance of the number of points in large balls is subvolume \cite{simoncoste, lachièzerey2026hyperuniformrandommeasurestransport}. This property holds for $L^d$-dependent perturbed lattices in dimensions $d=1,2$, but generally fails to hold in dimension $d\geq 3$ \cite{dereudre2024nonhyperuniformityperturbedlattices}. It is also satisfied by the Ginibre ensemble, the determinantal process with sine kernel, and the two-dimensional Coulomb gas \cite{huesmann2024linkhyperuniformitycoulombenergy, leblé2025twodimensionalonecomponentplasmahyperuniform}, among others.  
	
	A stationary (i.e., translation-invariant) point process $\xi$ of unit intensity in dimension $d=1$ is said to have a \textbf{Cyclic-Factor} if its law $\sfP\in\pps$ can be written as  
	\eq{
	\sfP = \int_0^1 \dif u~ \sfP_0\circ \theta_u^{-1},}
	where $P_0$ is not $\R$-stationary but is $\Z$-stationary, and the supports of the phases $\{\sfP_0\circ \theta_u^{-1}\}_u$ are mutually disjoint. This property was formally introduced in \cite{aizenman2000boundedfluctuationstranslationsymmetry} and provides a mathematical definition of crystallization, also referred to as (translational) symmetry breaking. Most results on crystallization remain conjectural, with only a few models rigorously proved to form crystals. This has been established for the one-dimensional Coulomb gas at any temperature \cite{Kunz74, AizenmanMartin, aizenman2000boundedfluctuationstranslationsymmetry}, and is expected for a wide range of statistical physics models, such as the hard-sphere model at high density, or the Coulomb gas in dimension $d\geq 3$ at low temperature. Notice how this representation resembles the law of the stationarized version of a deterministic lattice, which provides the most straightforward example of a crystal. This decomposition is also related to the extremal decomposition of Gibbs measures. Finally, we also note other interesting rigidity notions, such as Insertion-Tolerance and Deletion-Tolerance \cite{holroyd2013insertiondeletiontolerancepoint}.

	As seen from the results mentioned above, perturbed lattices, under assumptions on the perturbations’ moments, have often been used as tools to establish rigidity. Since a perturbed lattice can be viewed as a matching between the lattice and its perturbed version, there are natural connections with optimal transport, both in finite and infinite volume. Being an $L^p$-perturbed lattice is equivalent to having finite $p$-transport to a lattice or to Lebesgue measure. Optimal transport in the context of infinite point processes has been formalized in \cite{erbar2024optimaltransportstationarypoint}. Recently, using optimal transport tools and ideas, links were established between Hyperuniformity and perturbed lattices \cite{lachièzerey2025hyperuniformityoptimaltransportpoint, butez2024wassersteindistancehyperuniformpoint}. As stated in Theorem \ref{thm1}, we show that an $L^1$ dependent perturbed lattice is Number-Rigid and Cyclic-Factor (it is proved in \cite{dereudre2024nonhyperuniformityperturbedlattices} that this also implies Hyperuniformity). The main tools for this theorem rely on optimal transport, well-ordered matchings, and the ergodic theorem.  
	
	Our main application of Theorem \ref{thm1}, contained in Theorem \ref{thm2}, concerns non-singular Riesz gases. Recall that Riesz gases in dimension $d$, with homogeneity parameter $s$, form a broad class of infinite-particle systems with long-range, repulsive interactions. More precisely, two particles at positions $x,y\in\R^d$ interact via the potential $|x-y|^{-s}$ for $s>0$, $-|x-y|^{-s}$ for $s<0$, and $-\log|x-y|$ for $s=0$. We refer to \cite{LewinSurvey2022, serfaty2024lecturescoulombrieszgases} for general reviews of Riesz gases. These models include some of the most studied point processes, such as the Coulomb gas for $s=d-2$, the log-gas for $s=0$, $d=1$ (also called the Sine$_\beta$), and the Dyson sine process for $s=0$, $d=1$, $\beta=2$ (the infinite-volume limit of the Gaussian Unitary Ensemble).  
	
	In dimension $d=1$, the Coulomb gas with $s=-1$ (also called the one-dimensional One Component Plasma, or 1D Jellium) has been extensively studied and is fully solvable since the 1960s \cite{Baxter_1963, Kunz74, AizenmanMartin, aizenman2000boundedfluctuationstranslationsymmetry, Chafa__2022}. It is a well-known and elegant model in statistical physics, and it is known to crystallize at any temperature. For one-dimensional Riesz gases with $s<0$, it is explained in \cite{LewinSurvey2022} and demonstrated numerically in \cite{lelotte2023phasetransitionsonedimensionalriesz} that the expected phase diagram is depicted in Figure \ref{fig:Riesz1D}: crystallization would occur at any temperature for $s\leq -1$, whereas for $-1<s<0$ crystallization would occur only at sufficiently low temperature. In Theorem \ref{thm2}, we prove partially that conjecture, more precisely we prove that Riesz gases with $s\leq-1$ are Number-Rigid, Hyperuniform of type I, and Cyclic-Factor. Theorem \ref{thm1} may also help to establish the conjectured behavior for $-1<s<0$, which is still open to this day. The main ideas rely on the Fourier representation of the energy, and on generalizations of the one-dimensional Coulomb case.

	\begin{figure}
		\centering
		\includegraphics[width = 0.7\textwidth]{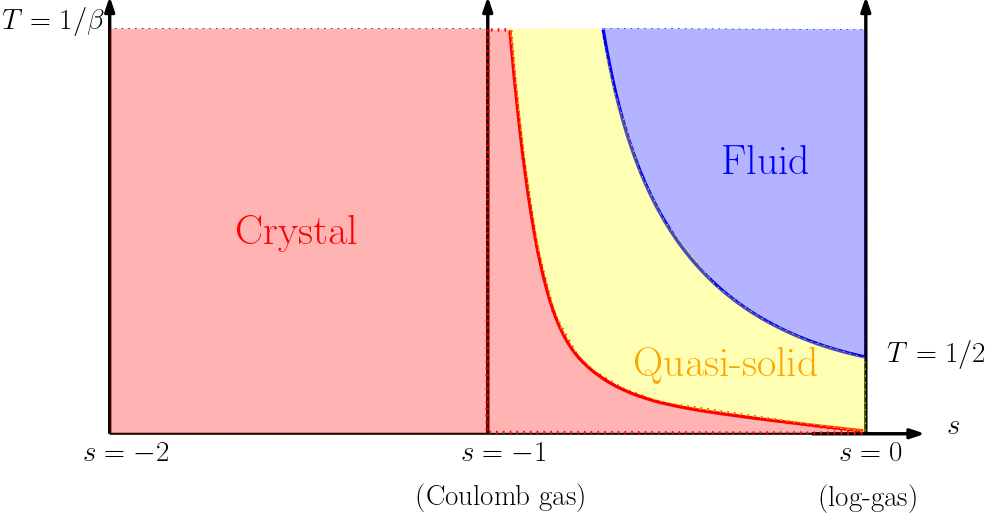}
		
		\caption{Conjectured phase diagram for the one-dimensional Riesz gas with $-2<s\leq 0$ \cite{LewinSurvey2022, lelotte2023phasetransitionsonedimensionalriesz}. The known and studied cases are as follows. The case $s=-1$ corresponds to the 1D Coulomb gas. The case $s=0$ corresponds to the Sine$_\beta$ gas. At $s=0$ and $\beta =2$, this is the Dyson sine process, a determinantal point process. It is expected that crystallization occurs at any temperature for $-2<s\leq -1$ (no phase transition), and that a double phase transition in temperature takes place for $-1<s<0$: from crystal, to quasisolid, and then to fluid as temperature increases. Theorem \ref{thm2} indicates crystallization for $s\leq -1$. Other conjectures remain open.}

		\label{fig:Riesz1D}
	\end{figure}
	
	The paper is organized as follows. In Section \ref{sec:SettingsTools}, we introduce the necessary theoretical background on point processes, stationarity, and perturbed lattices, which is mostly classical material. We also discuss optimal transport in the context of point processes. This section contains several auxiliary lemmas, such as the existence of a well-ordered matching. In Section \ref{sec:Results}, we present our two main theorems: the first is a general result on the rigidity of $L^1$ dependent perturbed lattices, while the second is an application to one-dimensional Riesz gases with parameter $-2<s\leq -1$, after formally defining them in infinite volume. The proof of the first theorem is detailed in Section \ref{sec:proofs}, and all proofs regarding Riesz gas are in Section \ref{sec:Riesz}. Proofs of result regarding optimal transport are in Section \ref{sec:OT}, and classical material regarding specific entropy is in Appendix \ref{sec:specificentropy}.
	
\tableofcontents

\section{Setting and tools} \label{sec:SettingsTools}

\subsection{Point processes} 
\label{ssec:pointprocesses}
A deterministic configuration of points in $\R$ is defined as a locally finite, integer-valued measure on $\R$, that is, a measure of the form $\xi = \sum_{j} \delta_{x_j}$ where $\{x_j\}_j$ is a finite or infinite sequence of points in $\R$ without accumulation points. In this framework, we allow multiple points to occur at the same location. The space $\G$ of such measures is a Polish space when endowed with the vague topology, under which the maps $\xi \mapsto \xi(f) = \int f(x)\,\xi(\dif x)$ are continuous for all compactly supported continuous functions $f$ on $\R$. We equip $\G$ with the $\sigma$-algebra generated by the counting maps $\xi \mapsto N_A(\xi) = \xi(A)$ for all Borel sets $A \subset \R$. 

A \textit{point process} is a random variable $\xi = \xi^\omega$ defined on a probability space $(\Omega,\trib,\P)$ with values in $\G$. By extension, we also use the term “point process” to refer to the law of this random variable, $\sfP = \P \circ \xi^{-1}$, as is standard in probability theory, and we denote by $\pp$ the set of such laws. 

The \emph{binomial point process}, denoted $\Binom_{n,A}$, with $n$ points in a finite volume $A$, is defined as the law of $\xi = \sum_{j=1}^n \delta_{X_j}$, where $\{X_j\}_j$ are i.i.d. random variables uniformly distributed in $A$. The \emph{Poisson point process}, denoted $\Poi$ when of unit intensity, can be seen as the natural extension of the binomial process to infinite volume. A key property of $\Poi$ is that for any pairwise disjoint Borel sets $A_1,\ldots,A_p$, the random variables $\xi(A_1),\ldots,\xi(A_p)$ are independent under $\Poi$. On a bounded Borel set $A$, the restriction $\Poi_A$ is defined as the law of a binomial process $(N,A)$ where $N$ is a Poisson random variable on $\N$ with parameter $|A|$. For comprehensive references and pedagogical introductions, see \cite{Last_Penrose_2017, daley2007introduction, kallenberg2017random}. The theory of point processes is naturally embedded in the more general framework of random measures, which we will also employ.

The space $\pp$, equipped with the weak topology (with respect to the vague topology), carries little information. We therefore define the \textit{local topology}, given by the fact that $\sfP_n$ converges locally to $\sfP$ if and only if 
\eq{\E_{\sfP_n}\crch{F(\xi)} \xrightarrow[n\to+\infty]{} \E_\sfP\crch{F(\xi)},\quad \forall F \text{ measurable, local, tame.}} 
A functional $F$ of measures is said to be \textit{local} if there exists a bounded box $\square$ such that $F(\xi)$ depends only on the configuration within that box. The same functional is said to be \textit{tame} if there exists some $C>0$ such that 
\eq{\abs{F(\xi)}\leq C(1+\xi(\square)).} 
In particular, the number of points in a bounded domain is a local and tame functional, and hence is captured by this topology—unlike with weak convergence (with respect to the vague topology).

\subsection{Stationarity and the tiling procedure} \label{ssec:tiling}
Translations $(\theta_x)_{x\in\R}$ act naturally on $\G$ via $\theta_x \left( \sum_j \delta_{x_j} \right) = \sum_j \delta_{x_j+x}$, and on $\pp$ via push-forward $\sfP \mapsto \sfP \circ \theta_x^{-1}$.
A point process $\sfP\in\pp$ is called \textit{stationary} if 
\eq{\sfP\circ\theta_x^{-1} = \sfP,\quad \forall x\in\R,} 
that is, $\xi\sim\sfP$ is stationary if and only if $\theta_x\xi$ has the same distribution as $\xi$ for all $x\in\R$. 
The space of such laws is denoted by $\pps$. 
Projections on Borel sets $A$ act naturally on these spaces, by 
$\xi\mapsto \xi_A = \pi_A(\xi) = \xi(A\cap\cdot)$ on $\G$, 
and $\sfP\mapsto \sfP_A = \sfP|_A= \sfP\circ\pi_A^{-1}$ on $\pp$.

Regarding Riesz gas, we will construct infinite volume by defining in finite domains and then going to the limit as the volume grows to infinity. To that end, it is useful to consider stationarized version of finite point processes. More precisely, given a point process $\sfP_\square$ in a finite square $\square$ of length $m$, we construct a stationarized version $\sfP_\square^{\mathrm{stat}}$ as follows. 
First, tile the space $\R$ using the box $\square$, as $\R = \bigcup_{k\in m\Z}\theta_k\square$, 
then on each translated box $\theta_k\square$ sample independent copies $\ens{\xi_k}$ with law $\theta_k\sfP_\square$. 
The \emph{tiled version} $\sfP^{\mathrm{til}}$ is the law of $\sum_{k\in m\Z} \xi_k$, which is $m\Z$-stationary. To make it $\R$-stationary, let $U$ be uniform in $\square$, independent from the $\ens{\xi_k}_k$. 
The \emph{stationarized law} $\sfP^{\mathrm{stat}}$ is defined as the law of $\sum_{k\in m\Z}\theta_U\xi_k.$

.
\subsection{Perturbed lattices} 
The shifted lattice $\xi = \sum_{j\in\Z} \delta_{j+U}$ with $U$ uniform on $[0,1]$ is a simple example of a stationary point process, which exhibits extreme rigidity. More generally, we refer to \emph{perturbed lattices} of the form $\xi= \sum_{j\in\Z}\delta_{j+U+p_j}$ with local perturbations $\{p_j\}$, possibly with dependence among the $p_j$. We assume that $\{p_j\}_{j\in\Z}$ is $\Z$-stationary, meaning $\{p_j\}_j = \{p_{j-k}\}_j$ in distribution, for all $k\in\Z$. In particular all perturbations share the same distribution. When the perturbations are independent, we speak of i.i.d. perturbed lattices. 

\subsection{Optimal transport}

The Kantorovich–Wasserstein distance between two deterministic measures $\xi,\eta$ is defined as 
\eq{\sfW_p^p(\xi,\eta) = \inf_{q} \int_{\R\times \R} |x-y|^p\dif q(x,y),} 
where the infimum runs over all couplings of $(\xi,\eta).$ It is defined as $+\infty$ if there is no coupling (if the two measures do not have the same mass). This setup can be extended to random measures when their laws are stationary. Optimal transport in the setting of stationary random measures is fully developed in \cite{erbar2024optimaltransportstationarypoint}. For $\sfP,\sfR \in\pps$ we can define 
\eq{\Wass_p^p(\sfP,\sfR) = \inf_{(\xi,\eta)} \E\crch{\inf_{q}\limsup_n \frac{1}{\abs{\square_n}}\int_{\square_n\times \R} |x-y|^p\dif q (x,y)}. \label{eq:OT}} 
Here the first infimum is over couplings at the level of probabilities, i.e., couplings of stationary point processes with $\xi\sim\sfP$ and $\e\sim\sfR$. The second infimum is at the level of configurations, with $q^\omega$ being a coupling of the measures $\xi^\omega$ and $\e^\omega$ for all $\omega$ (all these measures being of infinite mass). Lastly, the boxes are $\square_n = [-\frac{n}{2},+\frac{n}{2}]$, expanding to $\R$ as $n\to+\infty$. From \cite{erbar2024optimaltransportstationarypoint} we know that if $\Wass_p(\sfP,\sfR)<+\infty,$ optimal couplings exist.  

Let us assume we have a specific setup $(\Omega,\trib,\P)$ equipped with a flow $\theta:(x,\omega)\in\R \times\Omega\mapsto \theta_x\omega$ such that $\theta_{x+y} = \theta_x\circ\theta_y$ and $\theta_0 = \mathsf{id}$, with $\P$ satisfying $\P\circ\theta_x^{-1} = \P$ for all $x\in\R^d$. Stationarity of the law of $\xi$ is then implied by equivariance, defined as 
\eq{\xi^{\theta_x\omega}(A) = \xi^\omega(A-x), \quad \forall x\in\R,~\forall A \text{ Borel.}} 
This means that the flow is compatible with the natural flow of translations on the space.  In this specific setup, it is possible to rewrite 
\eq{\Wass_p^p(\sfP,\sfR) = \inf_{(\xi,\eta)}\inf_{q} \E\crch{\int_{\square_1\times \R } |x-y|^p\dif q (x,y)}.\label{eq:rewrite}} This fact is not trivial and corresponds to the proof of Proposition 1.2 in \cite{erbar2024optimaltransportstationarypoint}. Here, $(\xi,\e)$ and $q$ are all equivariant in some sense, even at the level of configurations for $q$, meaning that points must be matched in a translation-invariant way. The reason this works is that, under the hypothesis of equivariance, we can remove the limit superior from the expectation by the ergodic theorem. This equality holds on any fixed box $\square$, provided we divide accordingly by the volume $\abs{\square}.$ It is shown in \cite{erbar2024optimaltransportstationarypoint} that $\Wass_p$ is an extended geodesic distance on $\pps$. We can also define the same cost $\sfC_h$ for another cost at the level of points $h(x-y)$ (more general than $|x-y|^p$) under suitable assumptions on $h.$

We need a result on the Wasserstein distance for a stationarized point process.

\begin{lemma}\label{lem:Wass}
	Let $\square$ be a bounded square in $\R$ and let $\sfP_\square$ be a point process supported in $\square$. Let $\sfP_\square^{\mathrm{stat}}\in\pps$ be the stationarized version of $\sfP_\square$ . Then, denoting $\xi$ the variable for the measure $\sfP_\square$,
	\eq{\Wass_p^p(\sfP^{\mathrm{stat}}_\square,\Leb_{\R}) \leq \frac{\E_{\sfP_\square}\crch{\sfW_p^p(\xi,\Leb_\square)}}{\abs{\square}}.} The inequality is non-trivial when $\xi$ has the same mass as $\Leb_\square$, for instance if $\sfP(\dif\xi)$ charges configurations with $\abs{\square}$ number of points (and is integer).
\end{lemma}


Let us now define what we call \textbf{well-ordered} matching, which is specific to the one-dimensional case. We say a matching $q = (\mathsf{id},T)\xi = \sum_{x\in\xi}\delta_{(x,T x)}$ of two measures $\xi$ and $\e$ is well ordered, if 
\eq{x\leq y \implies T x \leq T y,~\text{ almost surely.}} 
For example, in the case of a perturbed lattice, if $\sfW_1(\sfP,\sfP^\Z)<+\infty$, then there exists a matching $q = \sum_{j\in\Z}\delta_{(j+U,j+U+p_j)}$ between $\xi = \sum_{j\in\Z}\delta_{j+U}$, with $U$ uniform on $[0,1]$, and $\eta = \sum_{j\in\Z}\delta_{j+U+p_j}.$ In this case being well-ordered means that 
\eq{\forall j\in\Z,~p_j\leq 1+p_{j+1},~\text{ almost surely.}} We can also put the uniform part $U$ inside the local perturbations $p_j$.

\begin{lemma}\label{lem:wo}
	Let $\sfP,\sfR\in\pps$ be stationary point processes such that $\Wass_1(\sfP,\sfR)<+\infty$. Then there exists a matching $q$ between $\xi\sim\sfP$ and $\e\sim\sfR$ which is well ordered and with $L^1$ perturbations. If $\sfR = \sfP^\Z$, then the matching gives a perturbed lattice $\xi= \sum_{j\in\Z} \delta_{j+p_j}$ which has law $\sfP$, with perturbations $\{p_j\}_j$ $\Z$-stationary and with $L^1$ moments. 
\end{lemma}

The proofs of both lemmas can be found in Section \ref{sec:OT}.

\subsection{Rigidities}

We introduce three notions of rigidity for point processes. A point process is said to be \textbf{Number-Rigid} if the number of points in any bounded domain depends only on the exterior configuration:
\eq{\forall A \text{ bounded Borel set},~  \xi(A) \text{ is almost surely a measurable function of } \xi(A^c\cap\cdot).}  
A point process is said to be \textbf{Hyperuniform} if
\eq{\lim_{r\to+\infty} \frac{\var \, \xi(B_r)}{|B_r|} = 0,}  
where $B_r$ denotes the centered ball of radius $r$. This contrasts with the Poisson point process, for which the limit is strictly positive.  

A stationary point process $\xi$ of unit intensity is said to have a \textbf{Cyclic-Factor} if its law $\sfP\in\pps$ can be written as
\eq{\sfP = \int_0^1 \dif u~ \sfP_0\circ \theta_u^{-1},} 
where $P_0$ is not $\R$-stationary but is $\Z$-stationary (i.e., $\sfP_0 \circ \theta_j^{-1} = \sfP_0$ for all $j \in \Z$), and the supports of the phases $\{\sfP_0\circ \theta_u^{-1}\}_u$ are mutually disjoint. The most straightforward example is the stationarized lattice $\Z$, whose law is given by $\sum_{j\in\Z}\delta_{j+U}$ with $U$ uniform on $[0,1]$, which is the most obvious example of a pure crystal. This definition is thought as a formalization of crystallization. To our knowledge, this definition first appears in \cite{aizenman2000boundedfluctuationstranslationsymmetry}.

There exists no direct implication between these rigidities. For instance, in \cite{klatt2020stronglyrigidhyperfluctuatingrandom}, the authors provide a great example which is Number-Rigid but not Hyperuniform. See also \cite{simoncoste} for a condition giving Number-Rigidity and Hyperuniformity. 

\section{Results}\label{sec:Results}

\subsection{Number-Rigidity and Cyclic-Factor}

It is known that $\sfP$ satisfying $\Wass_1(\sfP,\Leb_\R)<+\infty$, or equivalently $\Wass_1(\sfP,\sfP^{\Z})<+\infty$, where $\sfP^{\Z}$ is the law of $\xi = \sum_{j\in\Z}\delta_{j+U}$ with $U$ uniform on $[0,1]$, implies that $\sfP$ is Hyperuniform, see \cite{dereudre2024nonhyperuniformityperturbedlattices}. We add two rigidity properties in this setting. 

\begin{theorem}[Rigidities in one dimension]\label{thm1}
	In dimension $d=1$, let $\sfP\in\pps$ be the law of a stationary point process. Under the assumption 
	\eq{\Wass_1(\sfP,\sfP^{\Z})<+\infty,} 
	the point process $\sfP$ is Number-Rigid and has a Cyclic-Factor. 
\end{theorem}

Its proof is in Section \ref{sec:proofs}, see \ref{ssec:NR} for Number-Rigidity and \ref{ssec:CF} for Cyclic-Factor. 

\begin{rmk}[False contrapositive] 
	Both implications admit false contrapositives. The following example is Number-Rigid and has a Cyclic-Factor, but is not an $L^1$-perturbed lattice. Define $\sfP$ as the law of $\xi = \sum_{k\in\Z} N\delta_{Nk+U_N}$, where $N$ and $\{U_n\}_{n\in\N}$ are independent, each $U_n$ is uniform on $[0,n]$, and $N\in\N$ is distributed according to $\P(N=n)=a_n$ with $\sum a_n = 1$. The mean cost of transporting one particle of $\xi$ to $\Leb_\R$ is $N/2$. Thus, conditioning on $N$, 
	\eq{\Wass_1(\sfP,\Leb) \approx \sum_{n\in\N} a_n \tfrac{n}{2} = +\infty,} 
	for instance if $a_n\propto 1/n^2$. Trivially, $\xi$ is Number-Rigid. It also has a Cyclic-Factor, since its law coincides with that of $\xi' = \sum_{k\in\Z} N\delta_{kN+U+I_N}$ where $N$, $\{I_n\}$, and $U$ are independent, $U$ is uniform on $[0,1]$, and $I_n$ is uniform on $\{0,\ldots,n-1\}$. Alternatively, $U$ can be recovered as the fractional part of any particle coordinate in $\xi$. 
\end{rmk}

\begin{rmk}[No equivalent in higher dimensions]
	
	Is there other results of the form "finite $L^p$-distance to a lattice" implies some form of rigidity ? In general, proving that $\Wass_p(\sfP,\Leb_{\R^d})<+\infty$ shows that $\sfP$ is an $L^p$-perturbed lattice, with $\Z^d$-stationary perturbations. In higher dimensions, for instance as soon as $d\geq 3$, even the Poisson point process satisfies this, and is evidently not Number-Rigid, not Cyclic-Factor (it is not even Hyperuniform, by definition). In dimension $d=2$ the situation is less clear: a result of this form is that $L^2$-distance to a lattice is enough for Hyperuniformity \cite{dereudre2024nonhyperuniformityperturbedlattices}. The authors think that there exist counterexamples of $L^2$-perturbed lattices in dimension $d=2$ which are not Number-Rigid. Regarding crystallization (Cyclic-Factor), the authors expect that a result of this type is impossible in dimension $d=2$, for instance the zeros of the Gaussian Analytic Function (GAF) form a perturbed lattice with subgaussian moments \cite{sodin2005randomcomplexzeroesii}, but it is not crystallized (same argument goes for Ginibre). It seems specific to the one-dimensional case that one can recover this global uniform shift $U$, sign of crystallization, directly on a configuration $\xi$, $\text{P}(\dif\xi)$-almost surely. This is here done using the ergodic theorem.
\end{rmk}


\begin{rmk}[Alternate proof]
	There is an alternate proof of the Cyclic-Factor property, obtained by combining two results from the literature. First, as stated in \cite{aizenman2000boundedfluctuationstranslationsymmetry} (Theorem~2.1), if the discrepancies $\ens{\left|\xi(I)-|I|\right|}_{\text{interval } I}$ form a tight family of random variables, then $\sfP$ has a Cyclic-Factor. Second, as shown in \cite{huesmann2015tranportestimatesrandommeasures} (Theorem~1.3), if $\Wass_1(\sfP,\Leb_\R)<+\infty$, then the expectations $\ens{\E\crch{\abs{\xi(I)-|I|}}}_{\text{interval } I}$ are uniformly bounded, which implies tightness of the discrepancies. The proof of the result in \cite{aizenman2000boundedfluctuationstranslationsymmetry} relies on \textit{cocycles} in ergodic theory, and the proof of the result in \cite{huesmann2015tranportestimatesrandommeasures} uses the refined Campbell formula and the \textit{shift-coupling inequality}. Our proof only requires the existence of well-ordered matching and the ergodic theorem for $\Z$-stationary sequences of random variables. It is also simpler and more intuitive. Moreover it simultaneously establishes Number-Rigidity — a result that, to our knowledge, is absent from the existing literature. The authors were made aware of similar computations regarding 1D Coulomb in the nice paper \cite{Jansen_2014}.
\end{rmk}

\begin{rmk}[Optimal transport in dimension $d=1$]
In a way, Theorem \ref{thm1} proves that the geometric Wasserstein in dimension $d=1$ essentially captures point processes which are very rigid. The authors want to highlight another similar distance, specific to the one-dimensional setting, for stationary point processes, developed in \cite{huesmann2025gapmetricsstationarypoint}. This distance is based on comparing the gaps (the distances between consecutive points), and notably is used to prove uniqueness of minimizer of the free energy, in particular for Riesz gases with $s\in(0,1)$. There are also non-local Wasserstein distances, see \cite{huesmann2025nonlocalwassersteingeometrygradient}.
\end{rmk}

\subsection{Application to Riesz gases}
Riesz gases form a wide family of models, depending on the dimension $d$, the homogeneity parameter $s$, and the inverse temperature $\beta$. There exists a large literature on Riesz gases, which includes well-known special cases such as the Coulomb case $s=d-2$, log-gases, the 1DOCP, etc. We refer the reader to \cite{LewinSurvey2022,serfaty2024lecturescoulombrieszgases} for extensive accounts. In this paper, we restrict to negative homogeneity parameters $s<0$. 

We first define the Riesz gas in finite volume. For $s<0$, the Riesz potential is defined by 
\eq{g_s(x-y)=-|x-y|^{-s},} 
which is homogeneous, radial, and repulsive (note the minus sign, which is not present for $s>0$). Notice also that the potential is defined at the origin, these are the \textit{non-singular} Riesz gases. In a domain $\square_n = [-\frac{n}{2},+\frac{n}{2}]$ of finite volume $n$, the Hamiltonian of the Riesz gas with $n$ particles is defined as 
\eq{H_n^{(s)}(\g) = \frac{1}{2}\int\int_{\square_n\times\square_n} g_s(x-y) \nu(\dif x)\nu(\dif y),\quad \text{with }\nu = \g-\Leb_{\square_n}.} 
This energy consists of $n$ particles of positive unit charge interacting via the Riesz potential, together with a background of negative unit charge uniformly distributed on $\square_n$. The background exactly cancels the global charge, i.e.\ $\nu(\R) = \nu(\square_n) = 0$. Expanding gives 
\eq{H_n^{(s)}(\g) =
	\frac{1}{2}\sum_{x,y\in\g} g_s(x-y)
	-\sum_{x\in\g}\int_{y\in\square_n}g_s(x-y)\dif y
	+\frac{1}{2}\iint_{\square_n\times\square_n} g_s(x-y)\dif x\dif y,
	\label{eq:hamiltonian2}}
corresponding respectively to particle-particle, particle-background, and background-background interactions (the last one being a constant, included for simplicity). The Gibbs measure defining the Riesz gas in finite volume is 
\eq{\sfP_{n}^{s,\beta}(\dif\g) = \frac{1}{Z_n^{(s,\beta)}}~e^{-\beta H_n^{(s)}(\g)}~\Binom_{n,\square_n}(\dif\g),} where $Z_n^{(s,\beta)} =\int  e^{-\beta H_n^{(s)}(\g)}\Binom_{n,\square_n}(\dif\g)$ is a normalizing constant. We prove the existence of infinite-volume Riesz gases using specific entropy methods, which are presented in Appendix \ref{sec:specificentropy} for the sake of completeness. Recall that stationarized point processes, using the tiling procedure, are defined in Subsection \ref{ssec:tiling}, and that the local topology is defined in Subsection \ref{ssec:pointprocesses}.

\begin{theorem}[Existence of Riesz gas]\label{thm:existence}
	Let $-2<s<0$ and $\beta>0$. The sequence of stationarized point processes $\ens{\sfP_n^{s,\beta,\mathrm{stat}}}_n$ admits accumulation points for the local topology (which is stronger than the weak topology). A priori, uniqueness is not guaranteed. These accumulation points are called Riesz gases.
\end{theorem}

To the best of our knowledge, such an existence result regarding stationarized Riesz gases is not present in the literature. Our main theorem regarding Riesz gases for $-2<s\leq -1$ is as follows. It contains the case $s=-1$, which is the classical 1D Coulomb model (also called 1D one component plasma, or Jellium), a very studied and elegant model of statistical physics \cite{Baxter_1963,Kunz74, AizenmanMartin, aizenman2000boundedfluctuationstranslationsymmetry}.

\begin{theorem}[Rigidities of Riesz gases]\label{thm2}
	One-dimensional non-singular Riesz gases $\{\sfP_{s,\beta}\}$ with homogeneity parameter $-2<s \leq -1$, at any inverse temperature $\beta>0$, are $L^2$-perturbed lattices, 
	\eq{\Wass_2(\sfP_{s,\beta},\sfP^\Z)<+\infty.} 
	In particular, these non-singular Riesz gases are Hyperuniform of type I, meaning $\ens{\var_{\sfP_{s,\beta}} N_I}_{I \text{ interval}}$ is bounded. They are also Number-Rigid and have a Cyclic-Factor.
\end{theorem}

The proofs of these two theorems, and all details regarding these Riesz gases, can be found in Section \ref{sec:Riesz}. As mentioned in the introduction (see Figure \ref{fig:Riesz1D}), it is expected that a phase transition in temperature (between crystalline and quasi-solid regimes) occurs for the Riesz gas with $-1<s<0$, as discussed in \cite{LewinSurvey2022} and observed numerically in \cite{lelotte2023phasetransitionsonedimensionalriesz}. Such a result has not yet been proved rigorously. A central open question is therefore: how can one show that these Riesz gases form crystals at low temperature (depending on $s$)? Our conjecture is as follows. It would imply all three rigidity properties we introduced, with the Cyclic-Factor corresponding to crystallization.

\begin{conj}
	For $-1<s<0$, there exists $\beta_0(s)>0$, increasing in $s$, such that for $\beta>\beta_0(s)$ we have $\Wass_1(\sfP_{s,\beta},\Leb_\R)<+\infty$, whereas for $0<\beta<\beta_0(s)$ we have $\Wass_1(\sfP_{s,\beta},\Leb_\R)=+\infty$.
\end{conj}

Of course, it could be that the transition is trivial, meaning that $\beta_0(s) = +\infty$ and the distance would always be infinite. It would nonetheless be Cyclic-Factor at low temperature despite this (as the counterexamples). The authors think that this option is ruled out, as per the crystallization conjecture, the process should look like a perturbed version of $\Z$, and $L^1$ perturbations seems plausible (or any $L^p$).

\section{Proofs of Theorem \ref{thm1}}\label{sec:proofs}

\subsection{Number-Rigidity}\label{ssec:NR} 

This subsection contains the proof of the first part of Theorem \ref{thm1}, which states that a point process $\sfP$ satisfying $\Wass_1(\sfP,\sfP^\Z)<+\infty$ is Number-Rigid. The core idea is to use the ergodic theorem. It suffices to prove number-rigidity for domains of the form $D = [a,b]$ with $a<b$. Indeed, if $A$ is any bounded domain, choose an interval $D$ such that $A\subseteq D$. Then $N_A = N_D-N_{D\setminus A}$, where $N_D$ depends only on $\xi_{D^c}$ (hence only on $\xi_{A^c}$), and $N_{D\setminus A}$ also depends only on $\xi_{A^c}$.

Now let $D = [a,b]$ be the domain. Using Lemma \ref{lem:wo}, we can write $\xi = \sum_{j\in\Z}\delta_{j+p_j}$ with $L^1$ perturbations, and well-ordered particles, meaning $\ens{j+p_j}_j$ is an increasing sequence. Let
\al{I & = \min\ens{j\in\Z,~j+p_j>b},\\ J& =\max\ens{j\in\Z,~j+p_j<a}.}
These random variables are almost surely well-defined. Indeed, the first moment assumption on perturbations $\E|p_0|<+\infty$ implies
\eq{\sum_{j\geq 0} \P\pr{|p_0|\geq j-a} = \sum_{j\geq 0} \P\pr{|p_j|\geq j-a}<+\infty,}
since the $p_j$ all have the same law as $p_0$ by $\Z$-stationarity. By the first Borel–Cantelli lemma, $|p_j|\le j-a$, and thus $a\le p_j+j$ for all sufficiently large $j$, which implies that $J$ exists; the same argument gives the existence of $I$. The random variable $I$ (resp.\ $J$) gives the index of the particle immediately to the right (resp.\ left) of the domain $[a,b]$. Let $X_j = j+I+p_{j+I}$ for all $j\geq 0$, which are the particles strictly to the right of $[a,b]$, and similarly let $Y_j = j+J+p_{j+J}$ for all $j\leq 0$, which are the particles strictly to the left of $D$ (indexed negatively). We denote by $\mathscr{I}_\Z$ the $\Z$-invariant $\sigma$-algebra on $\R^\Z$, defined as the set of events $A$ of $\R^\Z$ such that $\theta_u A = A$ for all shifts $u\in\Z$. Then, applying the ergodic theorem twice,
\al{\frac{1}{n}\sum_{j=0}^{n-1}\pr{X_j-j}  = I+ \frac{1}{n}\sum_{j=0}^{n-1}p_{j+I} &\xrightarrow[n\to+\infty]{(a.s.)} I + \E\crch{p_0~|~\mathscr{I}_\Z} \\ 
	\frac{1}{n}\sum_{j=-n+1}^{0}\pr{Y_j-j}  = J+ \frac{1}{n}\sum_{j=-n+1}^{0}p_{j+J} &\xrightarrow[n\to+\infty]{(a.s.)} J + \E\crch{p_0~|~\mathscr{I}_\Z}.}
The ergodic theorem gives almost sure convergence, and the fact that $I$ and $J$ are random does not affect this (they are finite almost surely). 
Note that $\frac{1}{n}\sum_j (X_j-j)$ (resp.\ $\frac{1}{n}\sum_j (Y_j-j)$) is measurable with respect to the exterior of $[a,b]$, since it only uses the positions of particles strictly to the right (resp.\ left) of $[a,b]$ and subtracts their indices. Hence these averages and their limits are measurable with respect to $\xi_{[a,b]^c}$. Subtracting the two limits shows that $I-J$ is measurable with respect to the configuration outside $[a,b]$, and therefore the number of points in $[a,b]$ depends only on $\xi_{[a,b]^c}$. This proves number-rigidity. 

Note that the only place where we used the well-ordered matching is here: without it, $I-J$ could be unrelated to the number of points in the interval.

\subsection{Cyclic-Factor}\label{ssec:CF}

This subsection contains the proof of the second part of Theorem \ref{thm1}, which states that a point process $\sfP$ satisfying $\Wass_1(\sfP,\sfP^\Z)<+\infty$ is a Cyclic-Factor. We use the same core idea (ergodic averages) to identify the global shift. 

Using Lemma \ref{lem:wo}, we can write $\xi = \sum_{j\in\Z}\delta_{j+p_j}$ with $L^1$ perturbations, and well-ordered particles, meaning $\ens{j+p_j}_j$ is an increasing sequence. Let us set
\eq{I = \min\ens{j\in\Z,~j+p_j>0},}
and let $X_j = j+I+p_{j+I}$ for all $j\geq 0$, the particles strictly to the right of the origin. Then
\eq{\frac{1}{n}\sum_{j=0}^{n-1}(X_j-j) = I +\frac{1}{n}\sum_{j=0}^{n-1}p_{j+I}\xrightarrow[n\to+\infty]{(a.s.)} I + \E\crch{p_0~|~\mathscr{I}_\Z}.}
Hence the ergodic limit
\eq{\Phi(\xi) := \lim_{n\to+\infty} \frac{1}{n}\sum_{j=0}^{n-1}(X_j-j) = I + \E\crch{p_0~|~\mathscr{I}_\Z} =: I+A,\quad \text{almost surely,}}
which means we can recover the global shift $A = \E\crch{p_0~|~\mathscr{I}_\Z}$ by observing $\xi$. Denote the fractional part by $\ens{t} = t-\lfloor t \rfloor$. Then
\eq{\label{eq:recoverShift}\ens{\Phi(\xi)} = \ens{I+A} = \ens{A} =: U, \quad \text{almost surely,}}
because $I\in\Z$.

Let us examine how this behaves under translations. After shifting by $u\in\R$, the new index is
\eq{I^u = \min\ens{j\in\Z,~j+u+p_j\geq 0},}
and the corresponding particles are $X_j^u = j+I^u+p_{j+I^u}$ for $j\geq 0$. The ergodic theorem gives
\aln{\Phi(\theta_u\xi) & = \lim_{n\to+\infty}\frac{1}{n}\sum_{j=0}^{n-1} (X_j^u-j) \\ 
	&= \lim_{n\to+\infty} I^u+u+\inv{n}\sum_{j=0}^{n-1} p_{j+I^u} \\ 
	&= I^u+u+A.}
Thus
\eq{\Phi(\theta_u\xi) = \Phi(\xi)+u+I^u-I,~\forall u\in\R,}
almost surely. Since $I^u-I\in\Z$, we obtain
\eq{\label{eq:phiShift}\ens{\Phi(\theta_u\xi)} = \ens{\Phi(\xi)+u},~\forall u\in\R.}
By $\R$-stationarity of $\sfP$, the law of $\Phi(\xi)$ is invariant under translations, hence for every $u\in\R$,
\eq{\law\pr{\ens{u+A}}= \law\pr{\ens{u+\ens{A}}} = \law\pr{\ens{A}}.}
The following lemma shows that
\eq{U = \ens{A} \sim \mathrm{Uniform}_{[0,1[}.} 

\begin{lemma} Let $Z$ be a real random variable. Then
	\eq{\label{eq:uniform}\law(\ens{Z}) = \law(\ens{Z+u}) \forall u\in\R \iff\law(\ens{Z}) = \mathrm{Uniform}_{[0,1[}.}\label{lemma:uniform}
\end{lemma}
\begin{proof}
	The law of $Z$ (mod 1) is determined by its Fourier coefficients $\E\crch{e^{2i\pi k Z}}$ for $k\in\Z$. If there exists an irrational $u\in[0,1[$ such that $\law(\ens{Z}) = \law(\ens{Z+u})$, then
	\aln{\E\crch{e^{2i\pi k(Z+u)}} =\E\crch{e^{2i\pi k\ens{Z+u}}} = \E\crch{e^{2i\pi k \ens{Z}}} =  \E\crch{e^{2i\pi k Z}}.}
	Since $e^{2i\pi k u} \ne 1$ for any irrational $u$ and any $k\ne 0$, this forces $\E\crch{e^{2i\pi k Z}} = 0$ for all $k\ne 0$. With $\E\crch{e^{2 i\pi 0 Z}} = 1$, we conclude that $\ens{Z}$ is uniform on $[0,1)$. Conversely, the uniform law satisfies \eqref{eq:uniform} by $1$-periodicity.
\end{proof}

Conditioning on $U$ yields the phase decomposition of $\sfP$:
\eq{\sfP = \int_0^1\dif u ~\sfP(\cdot ~| ~U=u),}
where $\sfP(\cdot ~| ~U=u)$ denotes the conditional law given $U$, realized via a probability kernel $K(\dif\xi,u)$. The phases $\ens{\sfP(\cdot ~| ~U=u)}_u$ are mutually singular because the fractional part $\ens{\Phi(\xi)}$ identifies the phase, as seen in equation \eqref{eq:recoverShift}. The same argument shows that each phase $\sfP(\cdot ~| ~U=u)$ is not $\R$-stationary but is $\Z$-stationary (shifting by an integer does not change $U=\ens{A}$). Finally we claim
\eq{\sfP(\cdot ~| ~U=u) = \sfP_0\circ\theta_u^{-1},\quad \text{with } \sfP_0 = \sfP(\cdot ~| ~U=0).}
By definition, the kernel $K$ satisfies
\eq{\int f(\xi)g(\ens{\Phi(\xi)})\dif\sfP(\xi) = \int_{[0,1[}g(u)\dif u \int f(\xi)K(\dif\xi,u),}
for any bounded measurable $f$ and $g$. Extend $g$ and $K$ by $1$-periodicity when needed. Equivalently, $K$ is characterized by
\eq{\int f(\xi)g(\Phi(\xi))\dif\sfP(\xi) = \int_{[0,1[}g(u)\dif u \int f(\xi)K(\dif\xi,u),\label{eq:kernel}}
for any bounded measurable $f$ and any $1$-periodic bounded measurable $g$. Now,
\aln{\int f(\xi)g(\Phi(\xi))\dif\sfP(\xi) & = \int f(\theta_v\xi)g(\Phi(\theta_v\xi))\dif\sfP(\xi)  \nonumber \\
	&= \int f(\theta_v\xi)g(v+\Phi(\xi)) \dif\sfP(\xi)  \nonumber \\  
	& = \int_{[0,1[} g(v+u)\dif u \int f(\theta_v\xi)K(\dif\xi,u) \nonumber \\
	& = \int_{[v,1+v[} g(u)\dif u \int f(\theta_v\xi)K(\dif\xi,u-v).}
The first equality comes from the $\R$-stationarity of $\sfP$; the second equality uses \eqref{eq:phiShift} and the $1$-periodicity of $g$; the third is the kernel definition \eqref{eq:kernel}; the fourth is the change of variables $u\leftarrow u+v$. Lastly, denoting $K(\theta_{v}\dif\xi,u)$ as the pushforward of $K(\dif\xi,u)$ by the map $\theta_{v}$ for ease of notation, we obtain

\al{\label{eq:kernel2}	\int f(\xi)g(\Phi(\xi))\dif\sfP(\xi) & = \int_{[v,1+v[} g(u)\dif u \int f(\xi)K(\theta_{v}\dif\xi,u-v)\nonumber \\ 
& = \int_{[0,1[} g(u)\dif u \int f(\xi)K(\theta_{v}\dif\xi,u-v).} The last equality is implied by the $1$-periodicity of $g$ and $K$ with respect to the second variable. Putting equations \eqref{eq:kernel} and \eqref{eq:kernel2} altogether implies that \eq{K(\theta_{v}\dif\xi,u-v)= K(\dif\xi,u),\quad \forall u,v\in\R, \text{ a.e. }}
This completes the proof.

\section{Proofs for optimal transport}\label{sec:OT}

\subsection{Proof of Lemma \ref{lem:Wass}}

Let $\square_n$ be the square, centered on $\square$, consisting of $2n+1$ disjoint boxes which are translations of $\square$, such that $\abs{\square_n} = (2n+1) \abs{\square}.$ On each translated box $\theta_k\square$, let $\xi_k$ be a translated realization of $\sfP_\square$, with $\ens{\xi_k}_k$ independent. Let $q_k$ be an optimal coupling between $\Leb_{\theta_k\square}$ and $\xi_k$ (that optimal coupling exists, in finite volume). Let $q^{til} = \sum_{k\in m\Z}q_k$ be the tiled version, and if $U$ is uniform in $[0,1]$ and independent from the $\{\xi_k\}$, let $q^{stat}= (\theta_U,\theta_U)q^{til}$, which is an obvious random matching between a realization of $\Leb$ and $\sfP_\square^{stat}$. Using Corollary 2.7 from \cite{erbar2024optimaltransportstationarypoint}:
\aln{\Wass_p^p(\sfP_\square^{stat}, \Leb_{\R^d}) & \leq \liminf_{n\to+\infty}\frac{1}{\abs{\sqn}}~\E~\int_{\sqn\times\R^d}\abs{x-y}^p \dif q^{stat}(x,y) \\ & \leq\liminf_{n\to+\infty}\frac{1}{\abs{\sqn}}~\E~\int_{\theta_U \sqn\times\R^d} \abs{x-y}^p \dif q^{til}(x,y) 
	\\ &\leq \liminf_{n\to+\infty}\frac{1}{\abs{\sqn}}~\int_{u\in\square}\frac{\dif u}{\abs{\square}}\E~\int_{ \square_{n+1}\times\R^d} \abs{x-y}^p \dif q^{til}(x,y) \\ &\leq \liminf_{n\to+\infty}\frac{1}{\abs{\sqn}}~(2n+3)^d~\E~\int_{ \square\times\square} \abs{x-y}^p \dif q_0(x,y) \\ & \leq \liminf_{n\to+\infty}\frac{1}{\abs{\square}}~\frac{(2n+3)^d}{(2n+1)^d}~\E_{\sfP_\square}\crch{\sfW_p^p(\xi,\Leb_\square)} \\ &\leq \frac{\E_{\sfP_\square}\crch{\sfW_p^p(\xi,\Leb_\square)}}{\abs{\square}}.} In the third inequality, we used that $\theta_U\square_n$ is included in $\square_{n+1}$ and removed $U$ from the expectation. This concludes.

\subsection{Proof of Lemma \ref{lem:wo}}

To construct a well-ordered matching with finite $L^1$ cost, we replace the $L^1$ distance with a strictly convex cost $h$, which penalizes particle crossings. Assuming by contradiction that the optimal coupling for cost $h$ is not well-ordered, we can create, at least locally, another better coupling. The technical aspect is then to stationarize this coupling to find an invariant infinite-volume coupling, which is strictly better.

The Kantorovich–Wasserstein distance is defined for general cost function $h$, under suitable assumptions on $h$, by \eq{\sfC_h(\sfP,\sfR) = \inf_{(\xi,\eta)} \E\crch{\inf_{q}\limsup_n \frac{1}{\abs{\square_n}}\int_{\square_n\times \R} h(x-y)\dif q (x,y)}, \label{eq:OT2}} for two stationary point processes $\sfP,\sfR$. Let $h$ be any cost such that $|x-y|-1\leq h(x-y)\leq |x-y|$, radial, radially increasing, strictly convex and such that $h(0) =0$. For instance, the function $h(r) = \sqrt{1+r^2}-1$ fits this description. Then $\Wass_1(\sfP,\sfP^\Z)<\infty$ implies trivially $\sfC_h(\sfP,\sfP^\Z)<+\infty.$ Let $q = (\id,T)$ be an equivariant matching, optimal for $\sfC_h$. The existence of such a coupling, and that it is a matching, is ensured by Proposition 2.11 of \cite{erbar2024optimaltransportstationarypoint}. This matching also has $L^1$ perturbations. Let us prove that $q$ is well-ordered, using the strict convexity of $h$. Notice that we do not need $q$ to be optimal for $\Wass_1$ and we simply need a well-ordered matching. 

We prove here that the matching $q$ has to be cyclic monotone, in the following sense: \eq{\label{eq:cmonotone}\P-\text{a.s.}, \forall x,y\in \xi,~h(x-Tx)+h(y-Ty) \leq h(x-Ty)+h(y-Tx),} where we denoted $T$ the matching from $\xi$ to $\e.$ If $q$ does not satisfy this, then we build another equivariant matching $\tilde{q}$ which has strictly lower cost. There exists $\epsilon>0$ and a large box $\square$ such that \eqn{\P\pr{\mathcal{A}_\square}=\P\pr{\begin{aligned}
			\exists x,y\in\xi\cap\square,~Tx,Ty\in\e\cap \square \\
			h(x-Tx)+h(y-Ty)>\epsilon+h(x-Ty)+h(y-Tx)
	\end{aligned}} \geq \epsilon.} 
For the sake of simplicity, let us assume (at no cost) that $\square = \square_1$. Then we can define our new matching. Tile the whole space using $\theta_u\square_1$, meaning writing $\R = \bigcup_{k\in\Z}\theta_{k+u}\square_1$, where $u\in\square_1$ is fixed. On each  tile $\theta_{k+u}\square_1$, either the event $\mathcal{A}_{\theta_{k+u}\square_1}$ does not happen, and we do not change the matching, either it does happen, in which case we choose some $(x,y)$ in square $\square$ satisfying the event using a deterministic rule which does not depend on the tile (for instance, centering the box, and using the lexicographical order), and we exchange $x$ and $y$, meaning $x$ is now matched to $Ty$, and $y$ to $Tx$. This provides a new infinite volume matching $q_{til,u}$, which is still a coupling of $\xi$ and $\e$. Now notice that since $q^{\theta_k\omega}(A\times B) = q^\omega((A-k)\times(B-k))$ for all $k\in\Z$, and since our tiling is compatible with $\Z$ translates, this implies that $q_{til,u}$ is $\Z$-equivariant: \eqn{q_{til,u}^{\theta_k\omega}(A\times B) = q_{til,u}^\omega((A-k)\times(B-k)),\quad \forall k\in\Z,~\forall A,B.} Now, defining lastly \eqn{\tilde{q}^\omega = \int_{u\in\square_1} \dif u~q^{\omega}_{til,u}.} Then $\tilde{q}^\omega$ is still a coupling of $(\xi^\omega,\e^\omega)$, it is not an atomic measure or a matching since points are split but that does not matter, simply because optimal transport in \cite{erbar2024optimaltransportstationarypoint} is defined for random measures and couplings. By construction it is now $\R$-equivariant. Thus, formula \eqref{eq:rewrite} gives that \eqn{\E\int_{\square_1\times\R} h(x-y)\dif \tilde{q}(x,y)\leq \E\int_{\square_1\times\R}h(x-y)\dif q(x,y) -\epsilon^2.}

This inequality contradicts the optimality of $q$ for cost $\sfC_h$, thus $q$ is cyclic-monotone in the sense of \eqref{eq:cmonotone}.

In dimension $d=1$, this implies the matching is well-ordered, just as in classical optimal transport. Indeed, if $x<y$, let us assume by contradiction $Tx>Ty.$ Then, let $a = x-Tx,~b=y-Ty$ and $\delta = Tx-Ty>0$. Notice that $a<b$, $x-Ty= a+\delta = (1-t)a+tb$ and $y-Tx = b-\delta = ta+(1-t)b$ with $t= \delta/(b-a).$ Then strict convexity immediately contradicts the inequality $h(a)+h(b) \leq h(a+\delta)+h(b-\delta).$

Lastly, when the second law $\sfR$ is $\sfP^\Z$, the matching gives perturbations $\{p_j\}_j$ such that $\xi= \sum_j\delta_{j+p_j}$ has law $\sfP$. Its $\Z$-stationarity is explained in \cite{dereudre2024nonhyperuniformityperturbedlattices}, Section 1.4.

\section{Riesz gases}\label{sec:Riesz}

\subsection{Finite volume definition} \label{ssec:finiteVolume} 
Let us define the Riesz gas for $-2<s<0$. The Riesz potential for $s<0$ is defined by 
\eq{g_s(x-y)=-|x-y|^{-s}.} 
Its Fourier transform (convention $k\in \R\mapsto \int_\R g(x) e^{-i k x} \dif x$) exists as a tempered distribution, and is also radial and homogeneous, 
\eq{\widehat{g_s}(k) = \mathsf{finite~ part}~\pr{\frac{C_s}{\abs{k}^{1-s}}},} 
where the constant $C_s = \frac{\G(\frac{1-s}{2})}{-\G(\frac{s}{2})}\sqrt{\pi}~2^{1-s}$ is strictly positive. It is therefore called a \textit{positive-type} potential (recall that $\G(\frac{s}{2})$ changes sign at every non-positive even integer). This explains the constraint $s>-2$. See \cite{helgason1999radon} for details, in particular for the definition of the Hadamard finite part. 

In finite volume $\square_n = [-\frac{n}{2},+\frac{n}{2}]$ with $n$ particles, the Hamiltonian of the Riesz gas is defined as 
\eq{H_n^{(s)}(\g) = \frac{1}{2}\int\int_{\square_n\times\square_n} g_s(x-y) ~\nu(\dif x)~\nu(\dif y),\quad \text{with }\nu = \g-\Leb_{\square_n}.} 
This energy consists of $n$ particles of positive unit charge interacting via the Riesz potential, together with a uniformly distributed background of negative unit charge in $\square_n$. The background exactly cancels global charges, meaning $\nu(\R) = \nu(\square_n) = 0$. Expanding gives 
\eq{H_n^{(s)}(\g) = \sum_{x,y\in\g} g_s(x-y)-\sum_{x\in\g}\int_{y\in\square_n}g_s(x-y)\dif y +\frac{1}{2}\iint_{\square_n\times\square_n} g_s(x-y)\dif x\dif y,\label{eq:hamiltonian3}} 
corresponding respectively to particle-particle, particle-background, and background-background interactions (the latter being a constant, included for simplicity). 

As stated in p. 21 of \cite{LewinSurvey2022}, the Fourier representation of the Hamiltonian is given by 
\eq{H_n^{(s)}(\g) = \frac{C_s}{4\pi} \int_\R \frac{\abs{\widehat{\nu}(k)}^2}{\abs{k}^{1-s}}\dif k,\quad \text{with }\nu = \g-\Leb_{\square_n}.\label{eq:Fourier}} 
It is well defined because $\widehat{\nu}(k) = \nu(\R)+O(k) = O(k)$ near zero, which implies $\frac{\abs{\widehat{\nu}(k)}^2}{\abs{k}^{1-s}} = O\pr{\abs{k}^{1+s}}$ with $1+s>-1$. Moreover, $\abs{\widehat{\nu}(k)} \leq \abs{\nu}(\R) = 2n$, hence $\frac{\abs{\widehat{\nu}(k)}^2}{\abs{k}^{1-s}} = O\pr{\frac{1}{|k|^{1-s}}}$ as $\abs{k}\to\infty$, which is integrable since $s<0$. A non-trivial consequence is that the energy is positive, even though the pairwise potential is negative. 

Everything explained here also applies to any potential $g$, with the corresponding $\widehat{g}$ under suitable assumptions on $g$. For instance, it holds for $g(x-y) = e^{-\abs{x-y}}$ or $g(x-y) = e^{-(x-y)^2/2}$, both of which we will use. For instance, as soon as $g$ is continuous in $0$ and integrable at infinity, we also have \eq{\label{eq:energie_gen}H_{n,g}(\g) =\frac{1}{2}\iint_{\square_n\times\square_n} g(x-y)\nu(\dif x)\nu(\dif y) = \frac{1}{4\pi} \int_\R \widehat{g}(k)\abs{\widehat{\nu}(k)}^2\dif k.}

The Gibbs measure defining the Riesz gas in finite volume is 
\eq{\sfP_{n}^{s,\beta}(\dif\g) = \frac{1}{Z_n^{(s,\beta)}}~e^{-\beta H_n^{(s)}(\g)}~\Binom_{n,\square_n}(\dif\g),} where the partition function is \eq{Z_n^{(s,\beta)} = \int_\g e^{-\beta H_n^{(s)}(\g)}\Binom_{n,\sqn}(\dif\g).} The measure is well defined because of the following lemma.

\begin{lemma}\label{lem:partfunction}
	There exists a constant $C= C_{s,\beta}>0$ such that 
	\eq{e^{-Cn}\leq Z_n^{(s,\beta)} \leq 1.\label{eq:partitionFunctionBornes}}
\end{lemma}

\begin{proof}
	The upper bound follows directly from $H_n^{(s)}\geq 0$. For the lower bound, let $\square_n=\square_n^{(1)}\sqcup\ldots\sqcup\square_n^{(n)}$ where $\abs{\square_n^{(j)}}=1$, obtained by dividing $[-\frac{n}{2},+\frac{n}{2}]$ into $n$ unit intervals. Using the so-called "cell method", we have  
	\aln{Z_n^{(s,\beta)} &= \int e^{-\beta H_n^{(s)}}\Binom_{n,\square_n}(\dif\g) \\ 
		&= \frac{n!}{n^n} \int_{x_1<\ldots<x_n \in \square_n} e^{-\beta H_n^{(s)}}\dif x_1\ldots \dif x_n \\ 
		&\geq e^{-n}~\int_{x_1\in\square_n^{(1)},\ldots,x_n\in\square_n^{(n)}} e^{-\beta H_n^{(s)}} \dif x_1\ldots\dif x_n \\ 
		& \geq e^{-n} \exp\pr{-\beta \int_{x_1\in\square_n^{(1)},\ldots,x_n\in\square_n^{(n)}} H_n^{(s)} \dif x_1\ldots \dif x_n},} 
	where Jensen’s inequality is used in the last step. By integrating equation \eqref{eq:hamiltonian3}, we obtain 
	\aln{\int_{x_1\in\square_n^{(1)},\ldots,x_n\in\square_n^{(n)}} H_n^{(s)} \dif x_1\ldots \dif x_n 
		&= \sum_{j<k}\iint_{\square_n^{(j)}\times\square_n^{(k)}} g_s(x-y)\dif x\dif y-\sum_{j,k}\iint_{\square_n^{(j)}\times\square_n^{(k)}} g_s(x-y)\dif x\dif y \\ 
		&\qquad+ \frac{1}{2}\sum_{j,k}\iint_{\square_n^{(j)}\times\square_n^{(k)}} g_s(x-y)\dif x\dif y \\ 
		&= -\frac{1}{2}\sum_{j}\iint_{\square_n^{(j)}\times\square_n^{(j)}} g_s(x-y)\dif x\dif y \\ 
		&=  -\frac{n}{2}\iint_{[0,1]\times[0,1]} g_s(x-y)\dif x\dif y = -a_s n.} 
	Thus the inequality holds with $C = C_{s,\beta} = 1+\beta a_s.$
\end{proof}

\subsection{Infinite volume existence}

To define the Riesz gas in infinite volume, we seek accumulation points of the sequence of stationarized measures $\ens{\sfP_n^{s,\beta,\mathrm{stat}}}_n$ (as defined in Subsection \ref{ssec:tiling}). Using entropy methods, it suffices to prove 
\eq{\Ent{\sfP_n^{s,\beta}}{\Poi_{\square_n}}\leq C n,\label{eq:borneEntropie}} 
which follows easily:
\aln{ 
	\Ent{\sfP_n^{s,\beta}}{\Poi_{\square_n}} 
	&= \int \log\pr{\frac{\dif \sfP_n^{s,\beta}}{\dif \Binom_{n,\sqn}}~\frac{\dif \Binom_{n,\sqn}}{\dif \Poi_\sqn}} \dif \Poi_\sqn \\ 
	&= \int \log\pr{\frac{e^{-\beta H_n^{(s)}}}{Z_n^{(s,\beta)}}~\frac{n!}{e^{-n}n^n}} \dif \Poi_\sqn \\ 
	&= -\beta~\E_{\Poi_\sqn}\crch{H_n^{(s)}}-\log Z_n^{(s,\beta)} +\log\pr{e^n~\frac{n!}{n^n}} \\ 
	& \leq (C+1)n,} 
using respectively $H_n^{(s)} \geq 0$ (see \eqref{eq:Fourier}), $Z_n^{(s,\beta)}\geq e^{-Cn}$ (Lemma \ref{lem:partfunction}), and $\frac{n!}{n^n} \leq 1.$ 

Equation \eqref{eq:borneEntropie} and Lemma \ref{lem:ents} show that the specific entropy of the stationarized version of $\sfP_n^{s,\beta}$ satisfies 
\eq{\Ents{\sfP_n^{s,\beta,\mathrm{stat}}}{\Poi_\R} \leq \frac{\Ent{\sfP_n^{s,\beta}}{\Poi_{\square_n}}}{\abs{\square_n}} \leq C.} 
Recalling Proposition \ref{prop:ents}, the compactness of level sets for the specific entropy (with respect to the local topology) implies the existence of accumulation points for the sequence $\ens{\sfP_n^{s,\beta,\mathrm{stat}}}_n$. This concludes the proof of Theorem \ref{thm:existence}.

\subsection{The Coulombian case for $s=-1$}

For the Coulombian case $s=d-2=-1$, the model is very well known and explicitly computable, see \cite{Baxter_1963, Kunz74, AizenmanMartin, aizenman2000boundedfluctuationstranslationsymmetry,Chafa__2022}. See also \cite{lelotte2023phasetransitionsonedimensionalriesz,LewinSurvey2022} for a historical overview. It is one of the most famous, studied, and arguably beautiful models of statistical physics. It is known to crystallize at any temperature, is Number-Rigid (first observed in \cite{AizenmanMartin}), satisfies the DLR equations, etc.  

In this case, the energy has a special form, since the potential is $-|x-y|$, which is easily summable once particles are ordered. More precisely, if $x_{(1)}<\ldots < x_{(n)}$ are the ordered $n$ particles of $\g=\ens{x_1,\ldots,x_n}$, then the so-called Baxter equality 
\eq{-\sum_{1\leq i<j\leq n}|x_i-x_j| = -\sum_{1\leq i<j\leq n}(x_{(j)}-x_{(i)}) = \sum_{j=1}^n(2j-(n+1))x_{(j)},} 
implies that
\eq{H_{n}^{(s=-1)}(\g) = \sum_{j=1}^n(x_{(j)}-y_j)^2+\frac{n}{12},\label{eq:energieCoulomb}} 
where $y_j = j-\frac{n+1}{2}$ is the finite lattice of $n$ particles in $\square_n = [-\frac{n}{2},\frac{n}{2}].$ This energy being closely related to $\sfW_2^2(\g,\Leb_\sqn)$ implies (at least heuristically) that in infinite volume, the 1D Coulomb gas is an $L^2$-perturbed lattice.

\begin{proof}[Proof of Eq. \eqref{eq:energieCoulomb}]
	Starting from \eqref{eq:hamiltonian3}: \aln{H_n^{(s=-1)}(\g) & = -\sum_{1\leq j<k\leq n} \abs{x_{(j)}-x_{(k)}}  + \sum_{j=1}^n \int_{-n/2}^{n/2}\abs{x_{(j)}-y}\dif y -\frac{1}{2}\iint_{[-n/2,+n/2]^2}\abs{y-z}\dif y\dif z \\ &= \sum_{1\leq j<k\leq n} (x_{(j)}-x_{(k)}) +\sum_{j=1}^n \left( \int_{-n/2}^{x_{(j)}}(x_{(j)}-y)\dif y +\int^{n/2}_{x_{(j)}}(y-x_{(j)})\dif y \right) \\ & \qquad\qquad -\int_{-n/2}^{n/2}\int_{t}^{n/2}(z-y)\dif y\dif z  \\ & = \sum_{j=1}^n x_{(j)}(n-2j+1)+\sum_{j=1}^n \left(x_{(j)}^2+\frac{n^2}{4}\right)- \frac{n^3}{6} \\& = \sum_{j=1}^n \left( x_{(j)}^2 -2 x_{(j)}y_j+y_j^2\right)-\sum_{j=1}^n y_j^2+\frac{n^3}{4}-\frac{n^3}{6} \\ & = \sum_{j=1}^n(x_{(j)}-y_j)^2 -\frac{n(n^2-1)}{12}+\frac{n^3}{12} \\ & =\sum_{j=1}^n(x_{(j)}-y_j)^2 +\frac{n}{12},} which concludes the proof.
\end{proof}

\begin{rmk}
	Equation \eqref{eq:energieCoulomb} shows that $\sfP_n^{s=-1,\beta}$ is the law of $n$ independent points $x_{(j)}$ with Gaussian distribution $\mathcal{N}(y_j,1/2\beta)$, conditioned to be ordered. This point of view is used, for instance, in \cite{Chafa__2022}. It also provides a way to apply lattice Gibbs measure methods to establish the existence of the infinite-volume one-dimensional Coulomb gas, together with the Dobrushin–Lanford–Ruelle equations. 
\end{rmk}

\subsection{$L^2$ transport in finite volume}

We show that for $-2<s\leq -1$, Riesz gases define $L^2$-perturbed lattices. We first establish this in finite volume.

\begin{lemma}\label{lem:energievolumique}
	There exists $C=C_{s,\beta}$ such that 
	\[
	\E_{\sfP_n^{s,\beta}}\crch{\sfW_2^2(\g,\Leb_\sqn)} \leq Cn.
	\]
\end{lemma} 

\begin{proof}
	Let us use the notation $\lesssim$ for inequalities valid up to multiplicative strictly positive constants (depending only on $s,\beta$) and additive terms linear in $n$, the latter corresponding to volumic contributions. We will notably use Fourier formulas for the energy under consideration. First,
	\al{
		\sfW_2^2(\g,\Leb_\sqn) &\lesssim \sfW_2^2(\g,\textstyle\sum_{j=1}^{n}\delta_{y_j}) \quad (\text{where } y_j = j-\tfrac{n+1}{2}) \nonumber \\ 
		&\lesssim H_n^{(-1)}(\g) \quad (\text{cf. } \eqref{eq:energieCoulomb})\nonumber \\ 
		&\lesssim \int_\R \frac{\abs{\widehat{\nu}(k)}^2}{\abs{k}^{2}}\dif k\quad (\text{with } \nu = \g-\Leb_\sqn, \text{see } \eqref{eq:Fourier})  \nonumber \\
		& \lesssim \int_{\abs{k}\leq 1} \frac{\abs{\widehat{\nu}(k)}^2}{\abs{k}^{2}}\dif k+\int_{\abs{k}>1} \frac{\abs{\widehat{\nu}(k)}^2}{\abs{k}^{2}}\dif k  \nonumber \\ 
		&\lesssim \int_{\abs{k}\leq 1} \frac{\abs{\widehat{\nu}(k)}^2}{\abs{k}^{1-s}}\dif k+\int_{\abs{k}>1} \frac{2\abs{\widehat{\nu}(k)}^2}{1+\abs{k}^{2}}\dif k \quad \text{(since }-2<s\leq -1)  \nonumber \\
		&\lesssim H_n^{(s)}(\g)+H_{n,\exp}(\g) \label{eq:ineg1},
	}
	where $H_{n,g} = \frac{1}{2}\iint_{\sqn\times\sqn}g(x-y)\nu(\dif x)\nu(\dif y)$, recalling \eqref{eq:energie_gen} and the discussion at the end of Subsection \ref{ssec:finiteVolume}).  Let us now mention that if $g$ is integrable, then the second term in the energy, which is the particle-background interaction, is volumic: $\abs{\sum_{x\in\g}\int_{y\in\sqn}g(x-y)\dif y} \leq n\int_{y\in\R}g(y)\dif y.$ Recall that the background-background interaction is constant.
	
	For the second term in \eqref{eq:ineg1}, let $\square_n=\square_n^{(1)}\sqcup\ldots\sqcup\square_n^{(n)}$ with $\abs{\square_n^{(j)}}=1$, dividing $[-\frac{n}{2},+\frac{n}{2}]$ equally into $n$ intervals of unit length. Then
	\al{
		H_{n,\exp}(\g) & \lesssim \sum_{x\neq y\in\g} e^{-|x-y|}  \nonumber \\
		& \lesssim \sum_{j,k} \sum_{x\in\square_n^{(j)},~y\in\square_n^{(k)}} e^{-|j-k|+1}  \nonumber \\
		& \lesssim \sum_{j,k} N_j N_k e^{-|j-k|} \quad (\text{where } N_j = \g(\square_n^{(j)}))  \nonumber \\
		&\lesssim\sum_{j,k} (N_j^2+N_k^2)e^{-|j-k|}  \nonumber \\ 
		& \lesssim \sum_j N_j^2 \sum_k e^{-|j-k|}  \nonumber \\
		& \lesssim\sum_{j=1}^n N_j^2.\label{eq:ineg2}
	}
	We also have
	\al{
		H_{n,\text{gauss}}(\g) &\gtrsim \sum_{x\neq y\in\g} e^{-(x-y)^2/2} \nonumber \\
		& \gtrsim \sum_j \sum_{x,y\in\g\cap\square_n^{(j)},~x\neq y} e^{-1/2}  \nonumber \\
		& \gtrsim \sum_{j=1}^n N_j^2\quad (\text{where } N_j = \g(\square_n^{(j)})).\label{eq:ineg3}
	}
	Lastly, since $e^{-|k|^2/2} \leq \frac{C}{|k|^{1-s}}$ globally for $k\in\R$ and some constant $C$, 
		\eq{H_{n,\text{gauss}}(\g) = \int_\R \abs{\widehat{\nu}(k)}^2 e^{-\abs{k}^2/2}\dif k\lesssim H_n^{(s)}(\g).
	\label{eq:ineg4}}
	Putting equations \eqref{eq:ineg1}, \eqref{eq:ineg2}, \eqref{eq:ineg3}, \eqref{eq:ineg4} altogether imply
	\eqn{
		\sfW_2^2(\g,\Leb_\sqn) \lesssim H_n^{(s)}(\g).
	}
	Thus it suffices to show
	\eq{\label{eq:ineg5}
		\E_{\sfP_n^s}\crch{H_n^{(s)}} \leq Cn.
	}
	As is standard in statistical physics,
	\eqn{
		\E_{\sfP_n^s}\crch{H_n^{(s)}} = -\partial_\beta \log Z_n^{(s,\beta)}.
	}
	An application of Hölder's inequality shows that $\beta \mapsto \log Z_n^{(s,\beta)}$ is convex: indeed, if $\beta,\beta'\geq0$ and $0<t<1$,
	\aln{
		Z_n^{(s,t\beta+(1-t)\beta')} &= \int \pr{e^{-\beta H_n^{(s)}}}^t~\pr{e^{-\beta' H_n^{(s)}}}^{1-t} \dif \Binom_{n,\sqn} \\ 
		&\leq \pr{\int e^{-\beta H_n^{(s)}}\dif\Binom_{n,\sqn}}^t \pr{\int e^{-\beta' H_n^{(s)}}\dif\Binom_{n,\sqn}}^{1-t} \\ 
		& \leq \pr{Z_n^{(s,\beta)}}^t~\pr{Z_n^{(s,\beta')}}^{1-t}.
	}
	The function is also strictly decreasing from $0$ to $-\infty$, and recalling \eqref{eq:partitionFunctionBornes}, this gives
	\eqn{
		\partial_\beta\log Z_n^{(s,\beta)}\geq \frac{\log Z_n^{(s,\beta)}-\log Z_n^{(s,0)}}{\beta} \geq \frac{-Cn}{\beta},
	}
	which implies \eqref{eq:ineg5} and thus concludes the proof.
\end{proof}

\subsection{$L^2$ transport in infinite volume}

Recall that Riesz gases are defined as accumulation points of the sequence $\ens{\sfP_n^{s,\beta,\mathrm{stat}}}_n$, with respect to the local topology (see Theorem \ref{thm:existence}). Using Lemmas \ref{lem:energievolumique} and \ref{lem:Wass}, we obtain
\eqn{
	\Wass_2^2(\sfP_n^{s,\beta,\mathrm{stat}},\Leb_\R)\leq \frac{\E_{\sfP_n^{s,\beta}}\crch{\sfW_2^2(\xi,\Leb_{\sqn})}}{n}\leq C<+\infty.
}
Recall also that the local topology is stronger than the weak topology, and that the Wasserstein distance for point processes is lower semicontinuous with respect to the weak topology. It follows that a Riesz gas $\sfP^{s,\beta}$ satisfies
\eqn{
	\Wass_2^2(\sfP^{s,\beta},\Leb_\R)\leq \liminf_n \Wass_2^2(\sfP_n^{s,\beta,\mathrm{stat}},\Leb_\R)\leq C<+\infty.
}
Finally, the properties of Number-Rigidity and Cyclic-Factor follow from Theorem \ref{thm1}. Hyperuniformity of type I is a consequence of \cite{dereudre2024nonhyperuniformityperturbedlattices} (Appendix B). This concludes the proof of Theorem \ref{thm2}.

\appendix
\section{Specific entropy}\label{sec:specificentropy}

The space $\pps$ equipped with the weak topology (w.r.t. the vague topology) does not contain much information. We sometimes use what is called the \textit{local topology}, which is defined in Section \ref{sec:SettingsTools}. Specific entropy gives a tool to prove compactness with respect to this topology, and is defined in this section.

Recall that the relative entropy of two probability measures on a Polish space is defined by\eq{ \Ent{\mu}{\nu} = \begin{cases}
		\displaystyle \int \log\pr{\frac{\dif\mu}{\dif\nu}}\dif\mu & \text{if } \mu\ll\nu, \\ +\infty & \text{else.}\end{cases}} 
To define relative entropy on $\pps$, one would expect to simply use $\Ent{\sfP}{\Poi}$ to compare a stationary point process $\sfP$ with the Poisson point process. The problem is that this object is trivial in the sense that if $\sfP\in\pps$, $\sfP \ll \Poi \implies \sfP = \Poi.$ 
The better way is to only assume that $\sfP$ is stationary and $\sfP_A\ll\Poi_A$ for bounded domains $A$. This is usually the case, for instance if we are constructing $\sfP$ as a Gibbs measure. Then define \eq{\label{eq:Ents}\Ents{\sfP}{\Poi} = \lim_{n\to+\infty} \frac{\Ent{\sfP_\sqn}{\Poi_\sqn}}{\abs{\sqn}}. } The main results are as folows. We refer to \cite{georgii2011gibbs} (in a discrete setting) and \cite{georgiiZessin, dereudre2018introduction} (point process setting) for details.

\begin{proposition}\label{prop:ents}
	The limit in \eqref{eq:Ents} exists in $[0,+\infty]$ and is equal to \eq{\Ents{\sfP}{\Poi}= \sup_n\frac{\Ent{\sfP_\sqn}{\Poi_\sqn}}{\abs{\sqn}}.}  It is equal to $0$ iff $\sfP = \Poi$. It is affine and lower semi-continuous w.r.t. $\sfP$. Lastly, the level-sets $\ens{\sfP\in\pps,~\Ents{\sfP}{\Poi}\leq C}$ are sequentially compact for the local topology.
\end{proposition}

Unlike the distance $\Wass_p$ which is infinite as soon as we compare two point processes of distinct intensities, the specific entropy might be finite. 
Now, let us show the following classical lemma, which links infinite volume specific entropy with finite volume relative entropy, when stationarizing a finite point process. It enables to prove compactness for sequences of stationarized point processes.

\begin{lemma}\label{lem:ents}
	Let $\square$ be a bounded box in $\R^d$ and $\sfP$ be a point process living in $\square$. Let $\sfP_\square^{stat}\in\pps$ be the stationarized version of $\sfP_\square$. We have: \eq{\Ents{\sfP^{stat}_\square}{\Poi} \leq \frac{\Ent{\sfP_\square}{\Poi_\square}}{\abs{\square}}.}
\end{lemma}

\begin{proof}Let $\square_n$ be the square, centered on $\square$, consisting of $(2n+1)^d$ disjoint boxes which are translations of $\square$, such that $\abs{\square_n} = (2n+1)^d \abs{\square}.$ Let us denote $\til$ the map which consists in tiling. Then we have the following inequalities: \aln{\Ent{\sfP^{stat}_{\square}\big|_\sqn}{\Poi_{\square_n}} & = \Ent{\int_{u\in\square}\frac{\dif u}{\abs{\square}}~\sfP^{til}_\square\circ\theta_u^{-1}\circ\pi_{\square_n}^{-1}}{\Poi_{\square_n}} \\ &\leq \int_{u\in\square}\frac{\dif u}{\abs{\square}} ~\Ent{\sfP^{til}_\square\circ\theta_u^{-1}\circ\pi_{\square_n}^{-1}}{\Poi_{\square_n}} \\ &\leq \int_{u\in\square}\frac{\dif u}{\abs{\square}} ~\Ent{\sfP^{\otimes\Z^d}_\square\circ\til^{-1}\circ\theta_u^{-1}\circ\pi_{\square_n}^{-1}}{\Poi_\square^{\otimes\Z^d}\circ\til^{-1}\circ\theta_u^{-1}\circ\pi_{\square_n}^{-1}}. } We firstly used the convexity of relative entropy, and then the general properties of the Poisson point process. Now notice that at $\square_n$ and $u\in\square$ fixed, we do not need to generate every tile on the whole space. In fact, since we shift by $u$ and then restrict to $\square_n$, the largest information we might need is in the configuration in $\square_n + \square = \square_{n+1}$ which contains $k =(2n+3)^d$ tiles. Thus the entropy inside the last integral is the same as restricting to $k$ tiles \aln{\Ent{\sfP_\square^{\otimes k}\circ\til^{-1}\circ\theta_u^{-1}\circ\pi_{\square_n}^{-1}}{\Poi_\square^{\otimes k}\circ\til^{-1}\circ\theta_u^{-1}\circ\pi_{\square_n}^{-1}} & \leq \Ent{\sfP_\square^{\otimes k}}{\Poi_\square^{\otimes k}} = k~\Ent{\sfP_\square}{\Poi_\square}.} Here we used that the relative entropy decreases when taking pushforwards of measures (there is an exact formula). The last inequality implies the same bound integrating against $\dfrac{\dif u}{\abs{\square}}$, and then \aln{\frac{\Ent{\sfP^{til}_{\square}\big|_\sqn}{\Poi_{\square_n}}}{\abs{\square_n}}& \leq \frac{k}{\abs{\square_n}} \Ent{\sfP_\square}{\Poi_\square} = \frac{(2n+3)^d}{(2n+1)^d} \frac{ \Ent{\sfP_\square}{\Poi_\square}}{\abs{\square}}.} This gives the result for $n\to+\infty.$ 
	
\end{proof}

\section*{Acknowledgements}

David Dereudre and Rafaël Digneaux acknowledge the support of the CDP C2EMPI, together with the French State under the France-2030 programme, the University of Lille, the Initiative of Excellence of the University of Lille, the European Metropolis of Lille for their funding and support of the R-CDP-24-004-C2EMPI project. The authors also thank Martin Huesmann for useful discussions, as well as the anonymous referees for their helpful remarks.

\printbibliography

\end{document}